\scrollmode
\documentclass[11pt]{amsart}
\usepackage{amsrefs}
\usepackage{amssymb}
\usepackage[all,arc]{xy}
\usepackage[nointegrals]{wasysym}
\usepackage{graphicx}
\usepackage{mathrsfs}
\usepackage{tikz}
\usepackage{tikz-cd}
\usepackage{multicol}

\usepackage{hyperref}
\usepackage{cite}
\author{Simone Blumer\\ \tiny{Università degli Studi di Milano-Bicocca}}

\DeclareMathOperator{\image}{Im}

\DeclareMathOperator{\ind}{ind}
\DeclareMathOperator{\res}{res}

\newcommand{\F}{\mathbb{F}}

\usepackage{hyperref}

\DeclareFontFamily{U}{wncy}{}
\DeclareFontShape{U}{wncy}{m}{n}{<->wncyr10}{}
\DeclareSymbolFont{mcy}{U}{wncy}{m}{n}
\DeclareMathSymbol{\Sha}{\mathord}{mcy}{"58}
\DeclareMathSymbol{\sha}{\mathord}{mcy}{"78}

\title{Kurosh theorem for certain Koszul Lie algebras}
\date{\today}
\begin{document}
	\maketitle
	
	\newtheorem{thm}{Theorem}[section]
	\newtheorem*{thmA}{Theorem A}
	\newtheorem*{thmB}{Theorem B}
	\newtheorem*{thm*}{Theorem}
	\newtheorem{cor}[thm]{Corollary}
	\newtheorem{lem}[thm]{Lemma}
	\newtheorem{prop}[thm]{Proposition}
	\newtheorem{defin}[thm]{Definition}
	\newtheorem{exam}[thm]{Example}
	\newtheorem{examples}[thm]{Examples}
	\newtheorem{rem}[thm]{Remark}
	\newtheorem{case}{\sl Case}
	\newtheorem{claim}{Claim}
	\newtheorem{fact}[thm]{Fact}
	\newtheorem{question}[thm]{Question}
	\newtheorem{conj}[thm]{Conjecture}
	\newtheorem*{notation}{Notation}
	\swapnumbers
	\newtheorem{rems}[thm]{Remarks}
	\newtheorem*{acknowledgement}{Acknowledgement}
	
	\newtheorem{questions}[thm]{Questions}
	\numberwithin{equation}{section}
	
	\newcommand{\mc}[1]{\mathscr{#1}}
	\newcommand{\bu}{\bullet}
	\newcommand{\tor}{\operatorname{Tor}}
	\newcommand{\rad}{\operatorname{rad}}
	\newcommand{\set}[2]{\left\lbrace{#1}\ \big\vert\ {#2}\right\rbrace}
	\newcommand{\pres}[2]{\left\langle{#1}\ \big\vert\ {#2}\right\rangle}
	\newcommand{\invamalg}{\mathbin{\rotatebox[origin=c]{180}{$\amalg$}}}
	\newcommand{\ul}{{\mc U(\mc L)}}
	\newcommand{\li}{{\mc L}}
	\newcommand{\bbu}{{\bu,\bu}}
	\newcommand{\ext}{\operatorname{Ext}}
	\newcommand{\argu}{\hbox to 1.5ex{\hrulefill}} 
	\newcommand{\bbul}{{\bullet,\bullet}}
	\newcommand{\ten}{{T_\bullet}}
	\newcommand{\sym}{{S_\bullet}}
	\newcommand{\exa}{{\Lambda_\bullet}}
	\renewcommand{\hom}{\operatorname{Hom}}
	\newcommand{\As}{{A_{\bullet}}}
	\newcommand{\del}{\partial}
	\newcommand{\eps}{\varepsilon}
	\newcommand{\gr}{\operatorname{Gr}}
	\newcommand{\gen}[1]{\langle{#1}\rangle}
	\newcommand{\mf}[1]{\mathfrak{#1}}
	\newcommand{\abs}[1]{\vert{#1}\vert}

\begin{abstract}
The Kurosh theorem for groups provides the structure of any subgroup of a free product of groups and its proof relies on Bass-Serre theory of groups acting on trees. In the case of Lie algebras, such a general theory does not exists and the Kurosh theorem is false in general, as it was first noticed by Shirshov in \cite{shir}. However, we prove that, for a class of positively graded Lie algebras satisfying certain local properties in cohomology, such a structure theorem holds true for subalgebras generated in degree 1. Such class consists of Lie algebras, which have all the subalgebras generated in degree $1$ that are Koszul.
\end{abstract}
	
\tableofcontents
	\section*{Introduction}
	Consider a positively-graded Lie algebra $\li$ over a field $\F$. This means that $\li$ is a Lie algebra endowed with a vector space decomposition \[\li=\displaystyle\bigoplus _{i\geq 1}\li_i\] such that the Lie bracket is a graded $\F$-linear map $[\argu,\argu]:\li\otimes_\F \li\to \li$, i.e., $[\li_n,\li_m]\subseteq \li_{n+m}$.
	These Lie algebras are much well behaved than their ungraded counterparts, and they share many cohomological features with the class of pro-$p$ groups (cf. \cite{weig},\cite{cmp}). 
	The cohomology of a graded Lie algebra can be endowed with a second grading, the \textit{internal} one, that makes the cohomology algebra $H^\bu(\li,\F)$ into a bigraded algebra. 
	Among the $\mathbb N$-graded Lie algebras there is the class of the Koszul Lie algebras that are those Lie algebras $\li$ satisfying the following constraint on the bigraded cohomology algebra: 
	\begin{center}
		$H^{ij}(\li,\F)=0$ for $i\neq j$.
	\end{center}
	It is easy to see that Koszul Lie algebras are quadratic, in the sense that they have all of the generators of degree $1$ and the relations have degree $2$.
	In this paper we introduce a class of Lie algebras which have a more homogeneous Koszulity property, the Bloch-Kato Lie algebras. 
	A Lie algebra is \textit{Bloch-Kato} if it is Koszul, as well as any of its subalgebras that are generated in degree $1$.
	
	Since the cohomology ring of a Lie algebra is a graded-commutative algebra, it is natural to try and translate the above property for this class of algebras. 
	In particular, we see that the Bloch-Kato property of a Lie algebra is closely related with a Koszulity property on its cohomology, which was introduced in the commutative context by A. Conca \cite{conca}. 
	We say that a graded $\F$-algebra $A$ is \textit{universally Koszul} if every two-sided ideal $I$ of $A$ that is generated in degree $1$ satisfies:
	\begin{center}
		$\ext^{ij}_A(A/I,\F)=0$ if $i\neq j$.
	\end{center}
This is equivalent to requiring that the $A$-module $A/I$ has a linear free $A$-resolution. 
	
	We prove the following three main results.
	\begin{thmA}
		Let $\li$ be an $\mathbb N$-graded $\F$-Lie algebra, with cohomology ring $A=H^\bu(\li,\F)$. Then $\li$ is Bloch-Kato if, and only if, $A$ is universally Koszul.
	\end{thmA}
	\begin{thmB}
		Let $\mc A$ and $\mc B$ be two Bloch-Kato Lie algebras. Then their free product $\mc A\amalg\mc B$ is Bloch-Kato, too. 
	\end{thmB}
	From which follows:
	\begin{thm*}[Kurosh’ subalgebra theorem]
		Let $\mc A$ and $ \mc B$ be two Bloch-Kato Lie algebras, and let $\li=\mc A\amalg\mc B$ be their free product. Let $\mc H\leq \li$ be a Lie subalgebra generated by elements of degree $1$. Then $\mc H$ is a free product of a free Lie algebra and of some subalgebras of $\mc A$ and $\mc B$. More precisely, \[\mc H=\mc F\amalg\gen{\mc H_1\cap \mc A}\amalg\gen{\mc H_1\cap\mc B}\] where $\mc F$ is the free Lie algebra generated by any complement $W\leq \mc A_1\oplus\mc B_1$ with $W\oplus (\mc H_1\cap \mc A)\oplus(\mc H_1\cap\mc B)=\li_1$.
	\end{thm*}
	In fact, our original purpose was to prove Theorem B by applying some kind of Kurosh subalgebra theorem, just like in the group case, for which it is a standard argument. However, there is no such a result in the general Lie algebra case, as it was proved by Shirshov in \cite{shir}, who showed the existence of Lie algebras such that their free product contain a Lie algebra that is not free and is not isomorphic to the free product of any subalgebra of the factors and a free Lie algebra.  Eventually, we managed to prove Theorem B in a different way, and then we realized this could be used to prove the above Kurosh theorem for the distinguished class of Bloch-Kato Lie algebras.
	
	Eventually, notice that the example provided by Shirshov is not $\mathbb N$-graded, so that the main result of this paper sets in a different point of view.
\begin{acknowledgement}
	The author would like to thank his supervisors Conchita Mart\'inez-P\'erez, from the University of Zaragoza, Spain, and Thomas Weigel, from the University of Milano-Bicocca, Italy, for very helpful discussions concerning the (co)homolgical arguments used in the present paper. He would also express his gratitude towards the two institutions which the author belongs to, namely the University of Milano-Bicocca and the University of Zaragoza.
\end{acknowledgement}
	\section{Homological algebra}
	In this section we will introduce some tools from homological algebra in the graded context. For a complete treatment of the ungraded homological algebra, consult Weibel's book 
	By an $\mathbb N_0$-graded (associative) $\F$-algebra $A$ we mean a graded vector space $\bigoplus_{i\geq 0}A_i$ that is locally finite dimensional, i.e., $\dim_\F A_i<\infty$, and that is endowed with a graded $\F$-bilinear and associative product $A\otimes A\to A$ that we will denote by $(a,a')\mapsto aa'$. 
	All the algebras are assumed to be unital and connected, i.e., $A_0\simeq \F$. Denote by $A_+$ the augmentation ideal $A_+=\bigoplus_{i>0}A_i$, that is the kernel of the augmentation map $\varepsilon:A\to \F$. 
	The augmentation is graded and it induces a structure of $A$-module on $\F$, that is called the \textit{trivial module}. 
	A graded (left) $A$-module is a graded vector space $M=\oplus_i M_i$ that is also an $A$-module in the usual sense, for which the $A$-action $A\otimes M\to M$ is a graded $\F$-linear mapping.
	
	Henceforth, all the algebras (resp. modules) will be implicitely assumed to be $\mathbb N_0$-graded and connected (resp. $\mathbb N_0$-graded modules).
	
	If $V$ is a graded $\F$-vector space and $n\in \mathbb Z$, we denote by $V[n]$ the $n$-shifted vector space, i.e., $(V[n])_i=V_{i+n}$.
	\subsection{Cohomology of graded algebras}
	It is easy to see that the category of locally finite $A$-modules has enough projectives, for any free module is projective. 
	
	If $M$ and $N$ are two $A$-modules, we denote by $\ext^{\bu j}_A(M,N)$ the right derived functor of \[\hom^j_A(M,N)=\set{f:M\to N\text{ morphism of $A$-modules}}{f(M_k)\subseteq N_{k-j}}\]
	An element in $\hom_A^j(M,N)$ is called a morphism of degree $-j$. 
	
	It is straightforward to check that \begin{equation}\label{shiftExt}
		\ext^{ij}_A(M[m],N[n])=\ext^{i,j+m-n}_A(M,N)
	\end{equation} for graded $A$-modules $M$ and $N$ and integers $m,n$.
	
	The $\ext$-groups can be computed as follows: Let $P_\bu\to M$ be a projective $A$-resolution and apply the functor $\hom_A^j(\argu,N)$ to get the cochain complex \[\dots\to\hom^j_A(P_{i-1},N)\to\hom^j_A(P_{i},N)\to\hom^j_A(P_{i+1},N)\to\dots.\]
	Then \[\ext^{ij}_A(M,N)\simeq H^i(\hom^j_A(P_\bu,N))= \frac{\ker(\hom^j_A(P_{i},N)\to\hom^j_A(P_{i+1},N))}{\image({\hom^j_A(P_{i-1},N)\to\hom^j_A(P_{i},N)})}\]
	There is a distinguished projective resolution that one can always use, that is the so-called \textit{Bar complex} \[\dots\to A\otimes A_+^{\otimes n}\otimes M\to \dots \to A\otimes A_+\otimes A_+\otimes M\to A\otimes A_+\otimes M\to A\otimes M\to 0\]
	with differential given by \[\del (a_0\otimes\dots \otimes a_i\otimes m)=\sum_{s=1}^i (-1)^sa_0\otimes\dots\otimes  a_{s-1}a_s\otimes \dots\otimes a_i\otimes m+(-1)^{i+1}a_0\otimes \dots\otimes a_im.\]

	For an $\F$-algebra $A$ we call $H^\bu(A):=\ext^\bu_A(\F,\F)$ the \textit{cohomology} of $A$. 
	
	The following results give a description of the low degree $\ext$-groups (see \cite{pp}).
	\begin{prop}\label{prop:gen rel}
		Let $A$ be an $\F$-algebra, and $M$ an $A$-module. Then the following hold:
		\begin{enumerate}
			\item $\ext^{ij}_A(M,\F)=0$ for all $i>j$.
			\item There exist graded vector spaces $X$ and $Y$ such that $M$ is presented as \[A\otimes Y\to A\otimes X\to M\to 0\]
			and \begin{align*}
				X_j^\ast\simeq \ext^{0,j}_A(M,\F)\\
				Y_j^\ast\simeq \ext^{1,j}_A(M,\F)
			\end{align*}
		\item There exist graded vector spaces $V$ and $W$ such that $A$ is presented as an algebra as \[\ten(V)\otimes W\otimes \ten(V)\to \ten(V)\to A\to 0\]
		and \begin{align*}
			V_j^\ast\simeq \ext^{1,j}_A(\F,\F)\\
			W_j^\ast\simeq \ext^{2,j}_A(\F,\F)
		\end{align*}
		\end{enumerate} 
	\end{prop}
\begin{cor}\label{cor:gen module}
	Let $M$ be an $A$-module. If $\ext^{0,j}_A(M,\F)=0$, then $M$ has no relation of degree $j$, i.e., if $m_j\in M_j$ then there are elements $a_k\in A_k$ and $m_k\in M_k$ such that \[m_j=\sum_{i<j}a_im_{j-i}.\]
\end{cor}
	The vector space $H^\bu(A,\F)=\bigoplus_{i,j}H^{i,j}(A,\F)$ can also be endowed with the so-called \textbf{cup-product} $\smile$ that makes it into a bigraded (or $\mathbb N_0\times \mathbb N_0-$) algebra, i.e., \[\smile:H^{ij}(A,\F)\otimes H^{pq}(A,\F)\to H^{i+p,j+q}(A,\F).\]

	The following is the bigraded version of the classical cohomological long exact sequence theorem, and it can be deduced from the general long-exact sequence  theorem for (co)chain complexes (cfr. \cite{weib}).
	\begin{thm}\label{thm:longext}
		Let $A$ be an algebra. If $0\to L\to M\to N\to 0$ is an exact sequence of $A$-modules, then for all $j$ there is an induced long exact sequence \begin{align*}0\to& \ext^{0,j}_A(N,\F)\to \ext^{0,j}_A(M,\F)\to \ext^{0,j}_A(L,\F)\to\ext^{1,j}_A(N,\F)\to\dots\\
			\dots\to& \ext^{i,j}_A(N,\F)\to \ext^{i,j}_A(M,\F)\to \ext^{i,j}_A(L,\F)\to\ext^{i+1,j}_A(N,\F)\to\dots
		\end{align*}
	\end{thm}

	Also a version of the Eckmann-Shapiro Lemma holds for the internal degree of cohomology. In order to state such a result it is necessary to endow a tensor product with a grading. Let $A$ and $B$ be two $\mathbb N_0$-graded algebras, $M$ a graded $(A,B)$-bimodule and $N$ a graded left $B$-module. For $M$ being a graded $(A,B)$-bimodule means that $M=\bigoplus M_i$ is a graded vector space and both the action of $A$ and $B$ preserve the grading, i.e., $A_jM_i+M_iB_j\subseteq M_{i+j}$.
	 Then, the tensor product $M\otimes_B N$ is an $A$-module whose grading is given by \[(M\otimes_B N)_n=\left(\bigoplus_{i+j=n}M_i\otimes_\F N_j\right)\left/S_n\right.\] where $S_n$ is the vector subspace generated by elements of the form $xb\otimes y-x\otimes by$ for $x\in M_i$, $y\in N_j$ and $b\in B_{n-i-j}$.  
	\begin{thm}[Eckmann-Shapiro Lemma]\label{shap}
		Let $A$ be an algebra and let $B$ be a subalgebra of $A$. Suppose that $A$ is projective as a right $B$-module. Let $M$ be a left $B$-module and let $N$ be a left $A$-module.
		Then, \[\ext^{ij}_A(A\otimes _BM,N)\simeq \ext^{ij}_B(M,N).\]
		\begin{proof}
			Let $P_\bu\to M$ be a projective resolution of $M$ over $B$. Since $A$ is a projective right $B$-module, the sequence \[A\otimes_B P_\bu\to A\otimes_B M\]is a projective resolution over $A$. 
			
			Now, $\hom_A^j(A\otimes_B P_i,N)\simeq \hom^j_B(P_i,N)$, and hence $H^i(\hom_A^j(A\otimes_B P_\bu,N))\simeq H^i(\hom_B^j(M,N))$.
		\end{proof}
	\end{thm}
\subsection{Cohomology of Lie algebras}
Let $\li$ be an $\mathbb N$-graded $\F$-Lie algebra defined as in the Introduction. 
Its universal envelope (cfr. \cite{bou}) $\ul$ can be presented as the quotient of the tensor algebra on the vector space $\li$ by the ideal $J$ generated by the elements $x\otimes y-y\otimes x-[x,y]$ for every homogeneous elements $x,y\in\li$. 
By definition, the ideal $J$ is graded, and hence $\ul$ inherits a grading from $\li$ that makes it into a $\mathbb N_0$-graded associative algebra.

If $M$ is a $\li$-module (i.e., a graded $\ul$-module) one can define the cohomology of $\li$ with coefficients in $M$ as \[H^\bu(\li,M):=H^\bu(\ul,M)=\ext^\bu_\ul(\F,M).\] It follows that $H^\bu(\li,M)$ is bigraded. In this case, the cup product makes $H^\bu(\li,\F)$ into a graded-commutative algebra, i.e., $ab=(-1)^{nm}ba$ for $a\in H^n(\li,\F)$ and $b\in H^m(\li,\F)$.

In order to compute the Lie algebra cohomology one can use the Chevalley-Eilenberg complex $\hom_\F(\exa(\li),M)$ (cf. \cite{weib}), namely, \[H^\bu(\li,M)=H^\bu(\hom_\F(\exa(\li),M)).\] With respect to such a description, the cup product on $H^\bu(\li,\F)$ is induced by the natural map $\Lambda^n (\li)\otimes \Lambda^m(\li)\to \Lambda^{n+m}(\li)$.

	\subsection{Koszul algebras and modules}
	For an $\F$-vector space $V$, denote by $\ten(V)$ the tensor algebra over $V$, namely the vector space $\bigoplus_{i\geq 0}V^{\otimes i}$ where the product is given by concatenation of tensors.
	
	Let $A$ be an $\F$-algebra that admits a presentation \[A=\frac{\ten(V)}{(W)}\]
	where $V$ is a graded vector space and $(W)=\ten(V)\otimes W\otimes \ten(V)$ is the ideal generated by some $W\leq\ten(V)$.
	
	If $V$ is concentrated in degree $1$, then we say that $A$ is $1$\textbf{-generated}\footnote{Usually, in literature, one says that a certain object is "$n$-generated" whenever it is generated by $n$ of its elements. In this paper we are adopting such an expression for a different meaning.}.
	If moreover $W\leq V\otimes V$, one calls $A$ a \textbf{quadratic algebra}. In the latter case, one writes $A=Q(V,W)$.
	
	Henceforth, all the algebras we are dealing with are finitely generated, e.g., for a $1$-generated algebra $A$, $\dim_\F A_1<\infty$.
	
	From Proposition \ref{prop:gen rel}, one can see that $A$ is a $1$-generated algebra if, and only if, $\ext^{1,j}_A(\F,\F)=0$ for $j>1$, and $A$ is quadratic if, and only if, $\ext^{i,j}_A(\F,\F)=0$ for $i<j$, $i=1,2$. Recall that always $\ext^{ij}_A(\F,\F)=0$ for $j<i$.
	
	By extending this cohomological property to higher degrees, one recovers the definiton of a \textbf{Koszul algebra}: $A$ is Koszul if \[\ext^{i,j}_A(\F,\F)=0\quad \text{for } i\neq j.\]
	In that case, one says that the cohomology of $A$ is concentrated on the diagonal.
	
	 Similarly, let $M$ be an $A$-module with a presentation \[M=\frac{A\otimes H}{K}\]with $K$ an $A$-submodule of $A\otimes H$. 
	
	If $H$ is concentrated in degree $n$, then we say that $M$ is an $n$-generated $A$-module. 
	If moreover $A$ is $1$-generated and $K\leq A_1\otimes H$, one calls $M$ a \textbf{quadratic module} generated in degree $n$, or, simply, of degree $n$. In the latter case, one writes $M=Q_A(H,K)$.
	
	If $A$ is a quadratic algebra, from Corollary \ref{cor:gen module}, one has that $M$ is $n$-generated if, and only if, $\ext^{0,j}_A(M,\F)=0$ for $j\neq n$, and moreover $M$ is quadratic if, and only if, $\ext^{i,j}_A(M,\F)=0$ for $i\neq j-n$, $i=0,1$.
	
	Again, by extending this cohomological property to higher degrees, one recovers the definiton of a \textbf{Koszul $A$-module}: $M$ is a Koszul $A$-module if there is some $n\geq 0$ such that \[\ext^{i,j}_A(M,\F)=0\quad \text{for } i+n\neq j.\]
	It follows from \ref{shiftExt} that an $n$-generated $A$-module $M$ is Koszul if, and only if, $\ext^{ij}_A(M[-n],\F)=0$ for $j\neq i$, and hence we may only study $0$-generated Koszul $A$-modules. 
	
	Such modules are precisely those which admit a linear free resolution over $A$, i.e., a resolution by free $A$-modules $P_\bu\twoheadrightarrow M$ such that every $P_i=A\cdot (P_i)_i$ is generated in its $i$th degree (\cite{pp}).
	
	\medspace
	
	If $A=Q(V,W)$ is a finitely generated quadratic algebra, there is a kind of dual algebra that one can attach to $A$; this is the \textbf{quadratic dual} of $A$, and it is defined as \[A^!=Q(V^\ast,W^\perp)\] where $W^\perp\leq V^\ast\otimes V^\ast$ is the orthogonal complement of $W$ in $(V\otimes V)^\ast\simeq V^\ast\otimes V^\ast$, i.e.,
	\[W^\perp=\set{\alpha\in V^\ast\otimes V^\ast}{\alpha\vert_W\equiv 0}\]
	Notice that $(A^!)^!\simeq A$ for any (finitely generated) quadratic algebra $A$.
	
	The motivation for defining such a dual algebra is the following result:
	\begin{thm}[cf. \cite{pp}]\label{thm:diagext=dual}
		Let $A$ be a quadratic algebra. Then \[A^!\simeq\bigoplus _{i\geq 0} \ext^{i,i}_A(\F\,\F).\]
		
		In particular, $A$ is Koszul if, and only if, $A^!\simeq H^\bu(A,\F)$.
		
		Moreover, $A$ and $A^!$ are Koszul simultaneously.
	\end{thm}
Also, one can define a similar object associated with a module $M$ over a graded algebra. 
Suppose $M$ is a $0$-generated quadratic $A$-module, say $M=Q_A(H,K)$, and $A$ is a quadratic algebra. Then, we define the quadratic dual of $M$ as the $A^!$-module \[M^{!_A}=Q_{A^!}(H^\ast,K^\perp).\]
There holds:
\begin{thm}[cf. \cite{pp}]
	Let $A$ be a quadratic algebra and $M$ a quadratic module of degree $0$. Then \[M^{!_A}\simeq\bigoplus_{i\geq 0}\ext^{i,i}_A(M,\F).\]
\end{thm}

	We now introduce a Koszul property for the morphisms of graded algebras. We say that a morphism $f:A\to B$ is a \textbf{(left) Koszul homomorphism} if $B$ is a Koszul left $A$-module, or, equivalently, if the morphism can be completed into a linear free $A$-resolution of $B$
	\[\dots\to P_2\to P_1\to A\xrightarrow{f}B\to 0.\]
	The definition of right Koszul homomorphism is clear. Notice that, in particular, Koszul homomorphisms need to be surjective.
	\begin{prop}[Corollary I.5.4, p. 35, \cite{pp}]\label{koskos}
		Let $f:A\to B$ be a Koszul homomorphism. Then $A$ is a Koszul algebra if, and only if, $B$ is.
	\end{prop}
In fact, an algebra $A$ is Koszul if, and only if, the augmentation map $A\to \F$ is Koszul.
	\begin{lem}[Corollary I.5.9, p. 35, \cite{pp}]\label{lem:5.9}
		 Let $f:A\to B$ be a homomorphism of Koszul algebras. Then $f$ is left-Koszul if, and only if, $A^!$ is a free right $B^!$-module. If this is the case, then the dual morphism $B^!\to A^!$ is injective, and $B$ and $A^!/A^!B_+^!$ are dual Koszul modules over $A$ and $A^!$.
	\end{lem}
	\begin{prop}{\label{prop:constr ext}}
		Let $A$ be a Koszul $\F$-algebra, and let $L\xrightarrow{f}M\xrightarrow{\varepsilon}N\to 0$ be an exact sequence of Koszul $A$-modules of degree $0$.
		
		Then the following constraints hold:
		\begin{enumerate}
			\item $\ext^{i,j}_A(\ker\varepsilon,\F)=0$ for $j>i+1$, 
			\item $\ext^{i,j}_A(\ker f,\F)=0$ for $j>i+2$.
		\end{enumerate}
		In particular, $\ker f$ has no generator of degree $>2$.
		\begin{proof}
			(1) Consider the exact sequence $0\to \ker\varepsilon\to M\to N\to 0$. 
			It yields a long exact sequence \[\dots \to\ext^{ij}_A(M,\F)\to\ext^{ij}_A(\ker\varepsilon,\F)\to \ext^{i+1,j}_A(N,\F)\to\dots\]
			Since both $M$ and $N$ are Koszul modules, their $\ext$-groups are concentrated on the diagonal. 
			Then, for $j>i+1$, there holds $\ext^{ij}_A(M,\F)=\ext^{i+1,j}_A(N,\F)=0$, forcing $\ext^{ij}_A(\ker\varepsilon,\F)$ to vanish.
			
			(2) Similarly, the exact sequence $0\to\ker f\to L\to \ker \varepsilon\to 0$ induces the long exact sequence \[\dots \to\ext^{ij}_A(L,\F)\to\ext^{ij}_A(\ker f,\F)\to \ext^{i+1,j}_A(\ker\varepsilon,\F)\to\dots\]
			
			Let $j>i+2$. For (1), and since $L$ is Koszul, $\ext^{i+1,j}(\ker\varepsilon,\F)=\ext^{ij}_A(L,\F)=0$, forcing $\ext^{ij}_A(\ker f,\F)$ to vanish. 
		\end{proof}
	\end{prop}
	One may write the previous result with more generality on the degree of generation of the modules, although for our purpose the latter is enough.
	\begin{cor}{\label{inj koszul}}
		Let $f:L\to M$ be a morphism of Koszul $A$-modules of degree $0$. Suppose that $f$ is injective in all the degrees up to the second, and that $\text{coker} f$ is Koszul. Then $f:L\to M$ is injective (in every degree). 
		\begin{proof}
			By Lemma \ref{prop:constr ext}, $\ker f$ has no generator of degree $>2$. 
			But since $\ker f_i$ is trivial for $i=0,1,2$, it must be $0$ for all $i\geq 0$.
		\end{proof}
	\end{cor}
\subsection{Bloch-Kato Lie algebras}

In this subsection we introduce the notion of Bloch-Kato Lie algebra and show some examples.
\begin{defin}
	Let $\li$ be an $\mathbb N$-graded $1$-generated $\F$-Lie algebra. We say that $\li$ is Bloch-Kato if all of its $1$-generated subalgebras are Koszul.
\end{defin}
Notice that this is equivalent to requiring that the inclusion $\mc M\to \li$ induce a surjective algebra homomorphism $H^\bbu(\li)\to H^\bbu(\mc M)$, for every $1$-generated subalgebra $\mc M$ of $\li$. 

\medspace

\textbf{Examples.}

(1) Let $\mc F$ be a free Lie algebra generated by elements of degree $1$.
	We will see in Theorem \ref{NS} that the $1$-generated subalgebras of $\mc F$ is again a free Lie algebra. For the cohomology of a free Lie algebra over a space $V$ is (isomorphic to) the quotient $\ten V/(V\otimes V)=\F\oplus V$, it is clear that $\mc F$ is Bloch-Kato.

(2) Also, every $1$-generated abelian Lie algebra is Bloch-Kato, as its universal envelope -- as well as that of any of its $1$-generated subalgebras -- is a polynomial ring. However, these are not the only examples of Bloch-Kato Lie algebras, as the example shows.

(3)	Let $\li$ be the Lie algebra generated by $2d$ elements $x_1,y_1,\dots,x_d,y_d$ with the single relation $\sum_i[x_i,y_i]=0$. Put $A=\ul$.

The following is a \textit{free resolution} of $\F$ over $A$: \[ [F_\bullet\to \F]=\dots\to 0\to A\overset{d_1}{\to} A^{2d}\overset{d_0}{\to} A\overset{\varepsilon}{\to} \F\to 0\]
where \[d_0:(a_1,b_1,\dots,a_d,b_d)\mapsto \sum a_ix_i+b_iy_i,\]
\[d_1:a\mapsto a(y_1,-x_1,\dots,y_d,-x_d)\]
and \[\varepsilon:x_i,y_i\mapsto 0,\ 1\mapsto 1.\]
That resolution is minimal and linear, whence, $\mc L$ is Koszul with cohomological dimension $cd \mc L=2$. 

This also shows that \[\ext^2_A(\F,A)\simeq \F\] with the right action of $A$ induced by the augmentation map $\varepsilon:A\to \F$.
Moreover, $\ext^i_A(\F,A)=0$ for $i\neq 2$, whence $A$ is a \textit{Poincaré-duality algebra} of dimension $2$.
\begin{prop}[Poincaré duality algebras]
	Let $A\neq 0$ be an associative $\F$-algebra of type FP 
	satisfying \[\ext^i_A(\F,A)=H^i(A,A)=0,\ i\neq n.\] 
	Let $D$ be the right $A$-module $H^n(A,A)$.
	Then for every left $A$-module $M$ there holds the following homological duality \[H^i(A,M)\simeq \tor^A_{n-i}(D,M).\]
	\begin{proof}
		Notice that one has $cd A=n$ since $A$ is of type FP. 
		Let $P_\bullet=(P_i)_{0\leq i\leq n}$ be a finite projective resolution of $\F$ over $A$, i.e. the sequence \[0\to P_n\to P_{n-1}\to\dots\to P_0\to \F\to 0\] is exact and $P_\bullet$ is a finitely generated projective $A$-module. 
		Let $\bar P^{-i}=\hom _A(P_i,A)$ be the dual complex. Thus, $\bar P^\bullet$ is projective. 
		The sequence $\bar P^0\to\dots\to \bar P^{-n}$ is exact, as $H_{-i}(\bar P^\bullet)=H^i(A,A)=0$, for $i\neq n$. 
		Moreover, $\ext^n_A(\F,A)=H_{-n}(\bar P^\bullet)=\bar P^{-n}/im(\bar P^{-n+1}\to \bar P^{-n})$, and thus we have a projective $A$-resolution for $D=H^n(A,A)=\ext_A^n(\F,A)$ \[\dots\to \bar P^{-(n-1)}\to \bar P^{-n}\to D;\] more precisely, $\bar P^{n+\bullet}$ is the projective resolution.
		For every $i$ and every $A$-module $M$, $\hom_A(P_i,M)\simeq \bar P^{-i}\otimes_A M$, for $P_\bullet$ is finitely generated and projective. 
		Finally, \begin{align*}H^i(A,M)&=\ext^i_A(\F,M)=H_{-i}(\hom_A(P_\bullet,M))\\ 
			&\simeq H_{-i}(\bar P^\bullet\otimes_A M)=H_{n-i}(\bar P^{n+\bullet}\otimes_A M)\\
			&=\mbox{Tor}_{n-i}^A(D,M)\end{align*}
		i.e. \[H^i(A,M)=\ext^i_A(\F,M)\simeq\mbox{Tor}_{n-i}^A(D,M)\]
		
	\end{proof}
\end{prop}

It follows from the latter result that for our Lie algebra $\li$, it holds
\begin{align*}
	&\ext^0_A(\F,\argu)\simeq \mbox{Tor}_2^A(\F,\argu)\\
	&\ext^1_A(\F,\argu)\simeq \mbox{Tor}_1^A(\F,\argu)\\
	&\ext^2_A(\F,\argu)\simeq \mbox{Tor}_0^A(\F,\argu)=\F\otimes_A \argu
\end{align*}

Now, since the vector space dual functor $\argu^\ast=\hom_\F(\argu,\F)$ is exact on finite dimensional vector spaces, one has \begin{align}\label{torext}\mbox{Tor}_2^A(\F,\F)\simeq \ext_A^2(\F,\F)^\ast.\end{align}\\
Let now $\mc M$ be a proper Lie subalgebra of $\mc L$, and we show that $\mc M$ has cohomological dimension at most $1$, and thus it is free. This finally proves that $\mc L$ is Bloch-Kato.

In order to do that, we compute $H^2(\mc M,\F)$. 
For the isomorphism \ref{torext} holds, it is enough to compute $H_2(\mc M, \F)$. 
By applying the homological version of Eckmann-Shapiro Lemma (cf. \cite{weib}) to the subalgebra $B=\mc U(\mc M)$ of $A$, we get \[\mbox{Tor}_2^B(\F,\F)\simeq \mbox{Tor}_2^A(\F,A\otimes _B \F).\]
Therefore, for the duality relations of $A$, \[\mbox{Tor}_2^B(\F,\F)\simeq \ext^0_A(\F,A\otimes_B \F)=\hom_A(\F,A\otimes_B \F).\]
It follows from PBW theorem that the latter space is trivial, namely, the $A$-module $A\otimes _B\F$ has no $A$-fixed points. 
Indeed, let $\alpha:\F\to A\otimes _B \F$ be an $A$-linear map. For $\F$ is a simple $A$-module, the map $\alpha$ is determined by $\alpha(1)$. 
Complete an $\F$-basis $\{y_i,\ i\in I\}$ of $\mc M$ to an $\F$-basis $\{y_i,z_j,\ i\in I,\ j\in J\}$ for $\mc L$.
By PBW theorem, $B$ has $\F$-basis $\{y_{i_1}\cdots y_{i_n}\vert\ i_1,\dots, i_n\in I,\ n\in\mathbb N_0\}$ and it is infinite-codimensional in $A$, whose $\F$-basis can be suitably chosen to be $\{z_{j_1}\cdots z_{j_m}y_{i_1}\cdots y_{i_n}\vert\ i_1,\dots, i_n\in I,\ j_1\dots,j_m\in J,\ m,n\in\mathbb N_0\}$.
Note now that the induced module can be written as the following quotient of $A$:
\[A\otimes_B \F\simeq A/AB_+.\]
With respect to the above isomorphism, if $\alpha(1)=u+AB_+$, one can choose $u\in A$ to be of the form \[ u= \sum r_{j_1,\dots,j_r}z_{j_1}\cdots z_{j_r}\] for some $r_{j_1,\dots,j_r}\in \F$. Finally, \[0=\alpha(z_j\cdot 1)=z_j\alpha(1),\ j\in J\]
implies $\sum r_{j_1,\dots,j_r}z_jz_{j_1}\cdots z_{j_r}\in AB^+$, and thus the $r_{j_1,\dots,j_r}$'s need to vanish. \\

We conclude that $H^2(\mc M,\F)=0$, and thus $\mc M$ is free by Theorem \ref{freecoh}.

\section{Free products}
In this section, we recall the notion of free products of Lie algebras, and we state the graded version of the PBW Theorem. For a more detailed description, see \cite{bou}.

\begin{defin}
	Let $\mc A$ and $\mc B$ be two $\mathbb N$-graded Lie algebras over a field $\mathbb F$. Their free product is an $\mathbb N$-graded Lie algebra $\mc F$ endowed with two morphisms of graded Lie algebras $\iota_{\mc A}:\mc A\to\mc F$ and $\iota_{\mc B}:\mc B\to \mc F$, which satisfy the following universal property. For any pair of morphisms $\alpha:\mc A\to \li$ and $\beta:\mc B\to \li$ of graded Lie algebras, there is a unique morphism $\phi: \mc F\to \li$ of graded Lie algebras such that $\alpha=\phi\circ\iota_{\mc A}$ and $\beta=\phi\circ\iota_{\mc B}$.
\end{defin}

We also have a notion of free product of connected $\F$-algebras. In particular, if $A$ and $B$ are connected $\F$-algebras, we may define their free product $F=A\amalg B$ as follows:
\begin{enumerate}
	\item $F_0\simeq \F$, namely, $F$ is connected;
	\item $F_n=A_n\oplus [(A_{n-1}\otimes B_1)\oplus (B_1\otimes A_{n-1}) ]\oplus[(A_{n-2}\otimes B_2)\oplus (B_1\otimes A_{n-2}\otimes B_1)\oplus( B_2\otimes A_{n-2})]\oplus\dots\oplus B_n$ for $n>0$.
\end{enumerate}
The multiplication is given by concatenation followed, possibly, by the composition of elements that do belong to the same factor.

Eventually, there is a description for the free product of $\mathbb N$-graded Lie algebras $\mc A$ and $\mc B$. Let $\mc U(\mc A)$ and $\mc U(\mc B)$ be the universal envelopes of $\mc A$ and $\mc B$ respectively. Now, the free product of $\mc A$ and $\mc B$ can be described as the Lie subalgebra of $(\mc U(\mc A)\amalg\mc U(\mc B))_L$ generated by the images of $\mc A$ and $\mc B$. 

Recall the following important result.
\begin{thm}[Poincaré-Birkhoff-Witt]
	Let $\li$ be a $\mathbb N$-graded $\F$-Lie algebra. Let $\mathcal B$ be an ordered graded basis of $\li$ as a vector space. Then \[\mc\mathcal B=\set{v_1v_2\dots v_k}{v_i\in\mathcal B,\ v_i\leq v_{i+1}, k\geq 0}\] is a basis of $\ul$ as an $\F$-vector space. 
	
	Moreover, if $\mc M$ is a graded subalgebra of $\li$, then $\ul$ is a free $\mathbb N_0$-graded $\mc U(\mc M)$-bimodule.
\end{thm}

Let $\li$ be an $\mathbb N$-graded $\F$-Lie algebra, and denote by $\mc U(\li)$ its universal envelope. 
If $\mc H\subseteq \li$ is an $\mathbb N$-graded $\F$-Lie subalgebra, then by the Poincar\'e-Birkhoff-Witt theorem, $\mc U(\li)$ is a $\mathbb N_0$-graded free left $\mc U(\mc H)$-module. 
Let $M$ be a $\mathbb N_0$-graded left $\mc H$-module, and define \[\ind_{\mc H}^{\li}(M):=\mc U(\li)\otimes_{\mc U(\mc H)} M.\]
The functor $\ind_{\mc H}^{\li}(\argu)$ is called the \textbf{induction functor}. It is a covariant additive exact functor which is left adjoint to the restriction functor $\res_{\mc H}^{\li}(\argu)$, i.e. one has natural isomorphisms of bifunctors \begin{equation}{\label{ind}}
	\hom_{\li}(\ind_{\mc H}^{\li}(M),Q)\simeq \hom_{\mc H}(M,\res_{\mc H}^{\li}(Q)).
\end{equation}
The natural isomorphisms (\ref{ind}) are also called the Nakayama relations. 
It follows that the identity map $\F\to \F$ of the trivial $\mc H$-module $\F$ induces a homomorphism of graded left $\li$-modules \begin{equation}{\label{canmap}}
	\varepsilon^{\mc H}:\ind_{\mc H}^{\li}(\F)\longrightarrow \F.
\end{equation} 
\begin{thm}{\label{thm:mayervietoris}}
	Let $\li$ be an $\mathbb N$-graded $1$-generated $\F$-Lie algebra containing two $\mathbb N$-graded $\F$-subalgebras $\mc A$ and $\mc B$ satisfying $\li=\gen{\mc A,\mc B}$. 
	Then the following are equivalent: 
	\begin{itemize}
		\item[(i)] $\li\simeq \mc A\amalg\mc B$,
		\item[(ii)] $\ker\left(\varepsilon^{\mc A}-\varepsilon^{\mc B}:\ind_{\mc A}^{\li}(\F)\oplus\ind_{\mc B}^{\li}(\F)\to \F\right)\simeq \mc U(\mc \li)$.
	\end{itemize}
	\begin{proof}
		(i)$\Rightarrow$(ii) is Bass-Serre Theory for $\mathbb N$-graded Lie algebras (cf. \cite{cmp}).
		
		Suppose now that (ii) holds.
		Note that the canonical homomorphism $\phi:\mc F\to\mc L$ is surjective, where $\mc F=\mc A\amalg\mc B$. 
		By (ii), the sequence 
		\begin{equation}
			\label{mayervietoris}
			0\to \mc U(\li)\to \ind_{\mc A}^{\li}(\F)\oplus\ind_{\mc B}^{\li}(\F)\to \F\to 0
		\end{equation}
		is exact. 
		The same short exact sequence holds for $\mc F$, by Bass-Serre theory for graded Lie algebras. 
	By the Eckmann-Shapiro Lemma,
		\begin{align*}
			\ext_{\mc U(\li)}^\bu(\mc U(\li)\otimes_{\mc U(\mc A)}\F,\F)\simeq \ext^\bu_{\mc U(\mc A)}(\F,\F),\\
			\ext_{\mc U(\li)}^\bu(\mc U(\li)\otimes_{\mc U(\mc B)}\F,\F)\simeq \ext^\bu_{\mc U(\mc B)}(\F,\F).
		\end{align*}
		Thus we get a long exact sequences from Theorem \ref{thm:longext} induced by (\ref{mayervietoris})
		\begin{align*}
			\ext^1_{\mc U(\li)}(\mc U(\li),\F)\to \ext^2_{\mc U(\mc A)}(\F,\F)\oplus \ext^2_{\mc U(\mc B)}(\F,\F)\to \\
			\to\ext^2_{\mc U(\mc L)}(\F,\F)\to \ext^2_{\mc U(\mc L)}(\mc U(\mc L),\F)\to\dots
		\end{align*}
		But $\ext^i_{\mc U(\mc L)}(\ul,\F)=0$ for $i>0$, and thus the inflation $\inf:H^2(\li,\F)\to H^2(\mc F,\F)$ is an isomorphism $H^2(\mc L,\F)\simeq H^2(\mc A,\F)\oplus H^2(\mc B,\F)\simeq H^2(\mc U(\mc F),\F)$.
		
		Now, consider the short exact sequence associated with the surjection $\phi:\mc F\to \mc L$, \[0\to I\to\mc F\overset{\phi}{\to}\li\to 0.\]
		This yields the $5$-term exact sequence in cohomology
		\begin{equation*}
			\xymatrix{
				0\ar[r]&H^1(\li,\F)\ar[r]^{H^1(\phi)}&H^1(\mc F,\F)\ar[r]^\alpha&
				H^1(I,\F)^{\li}\ar[d]\\
				&&H^2(\mc F,\F)&H^2(\li,\F)\ar[l]\\}
		\end{equation*}
		But $H^i(\li,\F)\to H^i(\mc F,\F)$, for $i=1,2$, are the isomorphisms $H^1(\phi)$ and $\inf$, and thus $H^1(I,\F)^{\li}=0$. As $I $ is an $\mathbb N$-graded ideal of $\mc F $, one has $H^1(I ,\F)=\hom_{\mbox{\small{Lie}}}(I ,\F)=\hom_{\F}(I /[I ,I ],\F)$. 
		Thus, $H^1(I ,\F)^{\li }=0$ implies $I /[\mc F ,I ]=0$, and therefore $I =[\mc F ,I ]$. 
		Now, suppose $n=\min\set{m\geq 0}{I_m\neq 0}$ is finite. One has $I_n=[I ,\mc F ]_n=\sum_{1\leq j<n}[I_j,\mc F_{n-j}]=0$, since $I $ is a graded ideal, whence $I =0$, proving that $\ker(\phi)=I=0$.
	\end{proof}
\end{thm}
\begin{lem}{\label{cohom free prod}}
	Let $\mc A$ and $\mc B$ be two $\mathbb N$-graded finitely generated $\F$-Lie algebras. Then there is an isomorphism of graded algebras \[H^\bu(\mc A\amalg\mc B,\F)\simeq H^\bu(\mc A,\F)\sqcap H^\bu(\mc B,\F),\] where $\sqcap$ denotes the product in the category of graded connected algebras.
	\begin{proof}
		Let $\li=\mc A\amalg B$. Then, by Theorem \ref{thm:mayervietoris}, there is a short exact sequence of $\mc U(\li)$-modules \[0\to \mc U(\li)\to \ind_{\mc A}^{\li}(\F)\oplus\ind_{\mc B}^{\li}(\F)\to \F\to 0.\]
		The latter and the Eckmann-Shapiro Lemma induce a long exact sequence 
		\begin{align*}
			\dots\to\ext^i_{\mc U(\li)}(\mc U(\li),\F)\to \ext^{i+1}_{\mc U(\mc A)}(\F,\F)\oplus \ext^{i+1}_{\mc U(\mc B)}(\F,\F)\to \\
			\to\ext^{i+1}_{\mc U(\mc L)}(\F,\F)\to \ext^{i+1}_{\mc U(\mc L)}(\mc U(\mc L),\F)\to\dots
		\end{align*}
		and it follows that for $i\geq 1$, 
		\begin{equation}\label{prod coho}
			\ext_{\mc U(\li)}^i(\F,\F)\simeq \ext_{\mc U(\mc A)}^i(\F,\F)\oplus \ext_{\mc U(\mc B)}^i(\F,\F),
		\end{equation}
		or $H^i(\mc L,\F)\simeq H^i(\mc A,\F)\oplus H^i(\mc B,\F)$.
		
		Now, the natural inclusions of $\mc A$ and $\mc B$ into $\mc A\amalg\mc B$ induce the algebra homomorphisms $H^\bu(\mc A\amalg\mc B)\to H^\bu(\mc A)$ and $H^\bu(\mc A\amalg\mc B)\to H^\bu(\mc B)$. By the universal property of the direct product, we get an algebra homomorphism $H^\bu(\mc A\amalg\mc B)\to H^\bu(\mc A)\sqcap H^\bu(\mc B)$, that is an isomorphism, in the light of (\ref{prod coho}).
		
	\end{proof}
\end{lem}
	\section{Universally Koszul algebras}
	
	In this section we focus on the graded-commutative algebras, whose main example is the cohomology algebra of a Lie algebra. Recall that a graded algebra $A=\oplus_{i\geq 0}A_i$ is said to be \textbf{graded-commutative} if $ab=(-1)^{nm}ba$ for any homogenous elements $a\in A_{n},\ b\in A_m$.
		For a graded-commutative algebra $A$, define the \textbf{elementary skew-extension ring} $A[x]$ as the vector space $A\oplus Ax$, where $x$ has degree $1$, i.e., $A[x]\simeq A\oplus A[-1]$ as graded $\F$-vector spaces, endowed with the following product:
	\[(a_1+a_2x)(a_1'+a_2'x)=a_1a_1'+(a_1a_2'+(-1)^{\abs{a_1'}}a_2a_1')x.\]
	The elementary skew-extension algebra coincides thus with the free product in the category of the graded-commutative algebras of $A$ and $\exa(x)$, the free graded-commutative algebra on a single element $x$.
	
	Also, for a graded-commutative algebra $B$, define the \textbf{direct sum} $A\sqcap B$ as the categorical product of $A$ and $B$ in the category of graded-commutative connected algebras. 
	Explicitely, it is the connected algebra such that $(A\sqcap B)_i=A_i\oplus B_i$, $i>0$, and the multiplication is given component-wise.
	
\begin{defin}
		A graded-commutative algebra $A$ is said to be \textbf{universally Koszul} if $A\to A/I$ is a Koszul homomorphism for any ideal $I$ that is generated by elements of degree $1$.
\end{defin} In particular, since $A_+$ is a $1$-generated  ideal, the fact that $A\to A/A_+=\F$ is a Koszul homomorphism implies that $A$ is Koszul (cfr. Proposition \ref{koskos}).
	\begin{lem}{\label{lem:pol ring}}
		Let $A$ and $B$ be universally Koszul graded-commutative algebras. Then,
		\begin{enumerate}
			\item $A[x]$ is universally Koszul
			\item The direct sum $A\sqcap B$ is universally Koszul
		\end{enumerate}
		
		\begin{proof}
%
%
			
			See Proposition 30 in \cite{enhanced}.
	\end{proof}\end{lem}

	\section{Proofs of Theorems A and B}
	The first result of this section allows one to pass from Lie algebras to graded-commutative algebras and vice versa. 
	\begin{lem}\label{dual ue}
		\begin{enumerate}
			\item 	If $A$ is a quadratic graded-commutative algebra, then $A^!$ is the universal envelope of a quadratic Lie algebra.
			\item Conversely, if $\li$ is a quadratic Lie algebra, then $\ul^!$ is a (quadratic) graded-commutative algebra.
		\end{enumerate}
	\end{lem}
	\begin{proof}
		(1) Let $A=\exa(V^\ast)/(\Omega)$, where $\Omega\leq \Lambda_2(V^\ast)$. Choose a basis $(\omega_i)$ for $\Omega$, and elements $x_i,y_i\in V$ such that $\omega_i(x_j\otimes y_j)=\delta_{ij}$.
		
		Define the Lie algebra given by the following presentation:
		\[\li=\pres{V}{[v,v']-\sum_i\omega_i(v,v')[x_i,y_i]:\ v,v'\in V}.\]
		We want to show that $\ul$ is isomorphic to $A^!$.
		
		To this end, we present $A$ as a quadratic algebra $A=Q(V^\ast,R)$, where $R=\tilde\Omega\oplus \text{Span}_\F(\alpha\otimes \alpha:\alpha\in V)\leq V^\ast\otimes V^\ast$, and $\tilde \Omega$ is a lifting of $\Omega$ to a subspace of $V^\ast\otimes V^\ast$ with respect to the canonical projection $V^\ast\otimes V^\ast\to V^\ast\wedge V^\ast$.
		
		It is easy to see that $\ker(\ten(V)\to \ul)\subseteq\ker(\ten(V)\to A^!)$, as \[\omega_j\left( (v\otimes v'-v'\otimes v)-\sum_i\omega_i(v,v')(x_i\otimes y_i-y_i\otimes x_i)\right)=0\]
		Also the opposite inclusion holds true.
		
		
		(2) Since $\ul$ is quadratic, by Theorem \ref{thm:diagext=dual}, its dual is isomorphic to the diagonal subalgebra $\bigoplus_i\ext^{ii}_\ul(\F,\F)$ of the cohomology algebra $\ext^\bbu_\ul(\F,\F)=H^\bu(\li,\F)$, which is well-known to be graded-commutative.
	\end{proof}
	
	There follows Theorem A.
	\begin{thm}{\label{bk uk}}
		Let $\li$ be a graded Lie algebra with cohomology algebra $A=H^\bu(\li,\F)$.
		Then $\li$ is Bloch-Kato if, and only if, $A$ is universally Koszul.
		\begin{proof}

			(1) Let $\li$ be a Bloch-Kato Lie algebra. If $I_1\leq A_1=\li_1^\ast$, consider the $1$-generated Lie subalgebra \[\mc M=\pres{m\in\li_1}{i(m)=0,\ \forall i\in I_1}=\gen{I_1^\perp}.\]
			Since $\li$ is Bloch-Kato, the Lie algebra $\mc M$ is Koszul, and $H^\bu(\mc M,\F)\simeq \mc U(\mc M)^!$.
			
			By definition, $\res:H^\bu(\li,\F)\to H^\bu(\mc M,\F)$ is surjective and \[\ker \res_1=\set{\alpha\in\li_1^\ast}{\alpha\vert_{\mc M_1}=0}=I_1.\]
			Therefore, there is an ideal $J=I(I_1)\lhd A$ such that \[H^\bu(\li,\F)/I(I_1)\simeq B=H^\bu(\mc M,\F)\] is Koszul and $I(I_1)_1=I_1$. We want to show that $J$ is generated by $I_1$ as an ideal of $A$.

			Now, the projection $A\to B$ is Koszul, by Proposition \ref{lem:5.9}, since $A^!=\ul$ is a free right $B^!$-module, for $B^!=\mc U(\mc M)$ and the PBW Theorem. Thus, $B$ is a Koszul $A$-module.
			
			The exact sequence of $A$-modules $0\to J\to A\to B\to 0$ induces a long exact sequence involving the $\ext$ functor, that is,
			\[\ext^{0,j}_A(A,\F)\to\ext^{0,j}_A(J,\F)\to\ext^{1,j}_A(B,\F)\to\ext^{1,j}_A(A,\F)\]
			Notice that $J_0=0$, for $B_0\simeq A_0$.
			
			For $B$ being a Koszul $A$-module it means that $\ext^{\bu,\bu}_A(B,\F)$ is concentrated on the diagonal, and hence the long exact sequence shows that $J$ is generated in degree $1$, namely $J=AI_1$.
			(More precisely $\ext^{i,j}_A(J,\F)=0$ for $j\neq i+1$, and $J$ is a Koszul $A$-module.)
			
			(2) Suppose $A$ is a universally Koszul graded-commutative algebra. 
			Then $A$ is Koszul and $A^!=\mc U(\mc A)$ is the universal envelope of a Koszul Lie algebra $\mc A$ generated by $A_1^\ast$, by Lemma \ref{dual ue}.
			Let $V\leq A_1^\ast$, and $\mc B=\gen{V}$ be the Lie subalgebra of $\mc A$ generated in degree $1$ by $V$.
			Let $I$ be the ideal of $A$ generated by $V^\perp\leq A_1$ in degree $1$.
			
			Thus $A/I$ is Koszul since $A$ is universally Koszul, and $(A/I)^!$ is a quadratic subalgebra of $A^!$ by Lemma \ref{lem:5.9}.
			It follows that $(A/I)^!=\mc U(\mc N)$ for some (quadratic) Lie algebra $\mc N$ generated by $(A_1/I_1)^\ast=V$. 
			But, by Lemma \ref{lem:5.9}, $\mc N$ is a Lie subalgebra of $\mc A$, and hence, $\mc N=\mc B$, as they have the same generating space. 
			Thus, $\mc B$ is Koszul and \[H^\bu(\mc B,\F)\simeq A/I.\]
			
			Eventually, notice that since $A=H^\bu(\li,\F)$ is $1$-generated, $\mc U(\li)$ is a Koszul algebra, and hence $A^!=\mc U(\li)$, namely $\mc A=\li$.
		\end{proof}
	\end{thm}
	
	Since the direct sum of universally Koszul algebra is universally Koszul, by Lemma \ref{lem:pol ring}, the proof of Theorem B follows immediately.
	\begin{thm}\label{et=>bk}
		The class of Bloch-Kato Lie algebras is closed under taking free products.
		\begin{proof}
			Let $\mc A$ and $\mc B$ be two Bloch-Kato $\F$-Lie algebras. Then their cohomology rings $A=H^\bu(\mc A,\F)$ and $B=H^\bu(\mc B,\F)$ are universally Koszul by Theorem A.
			Now, $H^\bu(\mc A\amalg\mc B,\F)=A\sqcap B$ by Lemma \ref{cohom free prod}, and hence it is universally Koszul by Lemma \ref{lem:pol ring}. 
			Therefore, $\mc A\amalg \mc B$ is Bloch-Kato by Theorem A.\end{proof}
	\end{thm}
	
	\section{Kurosh subalgebra theorem}
	
	In the realm of (pro-$p$) groups, the Kurosh theorem gives the structure for an arbitrary subgroup of the free product of groups. This result is an immediate consequence of the Bass-Serre theory for groups acting on trees, and bacause of this geometric setting, its proof appears very elegant.
	
	Unfortunately, for the (ungraded) Lie algebras such a result is not available, and indeed it is false (see \cite{shir}.)
	However, for the positively-graded case, things behave better, and more similar to the group case, even though a complete Bass-Serre theory is not available; in fact, there is no action of a Lie algebra on graphs. 
	
	To this end we begin by proving a Nielsen-Schreier theorem for Lie algebras. It can be seen as a Kurosh' theorem for free $\mathbb N$-graded Lie algebras freely generated in degree $1$. 
	\begin{thm}\label{NS}
		Let $\mc M \subseteq \li\langle V\rangle$ be a $1$-generated subalgebra of the free $\F$-Lie algebra $\li\langle V\rangle$, $V$ an $\F$-vector space. Then $\mc M \simeq \li\gen{\mc M_1}$ is a free Lie algebra.
		\begin{proof}
			It suffices to show that the canonical $\F$-Lie algebra homomorphism $\phi:\li\gen{\mc M_1}\to \mc M $ is injective. 
			Let $\mc X$ be a basis for $V$ such that $\mc X\cap\mc M_1$ is a basis for $\mc M_1$. 
			Let $\mc B $ be the Hall-basis made with respect to an ordering on $\mc X$, and let $\mc B ^{\mc M}$ be the subset of $\mc B $ made up from the basis $\mc M_1\cap \mc X$ with the induced ordering. 
			Then $\text{Span}_\F(\mc B ^{\mc M})=\text{im}(\phi)$, and as $\mc B ^{\mc M}\subseteq \mc B $, one concludes that $\mc B ^{\mc M}$ is a linearly independent set. 
			Thus $\phi$ must be injective.
		\end{proof}
	\end{thm}
	We also prove the following characterization of the free Lie algebras, that the analogue of Stallings-Swan theorem for groups of cohomological dimension $1$.
	\begin{thm}{\label{freecoh}}
		Let $\li$ be a $1$-generated $\mathbb N$-graded $\F$-Lie algebra. Then $H^2(\li ,\F)=0$ if, and only if, $\li \simeq \li\gen{\li_1}$ is a free $\F$-Lie algebra.
		\begin{proof}
			If $\li $ is free, then $H^2(\li ,\F)=0$, since $0\to \mc U (\li )\otimes \li_1\to \mc U (\li )\to \F\to 0$ is a projective resolution of the trivial $\mc U (\li )$-module $\F$. \\
			Suppose $H^2(\li ,\F)=0$. It suffices to show that the canonical map $\phi :\li \gen{\li_1}\to \li $ is injective. Let $I =\ker(\phi )$, and put $\mc F =\li \gen{\li_1}$. Then the inflation-restriction sequence (namely, the $5$-term sequence induced by the Hochschild-Serre spectral sequence) associated with the exact sequence of $\mathbb N$-graded $\F$-Lie algebras $0\to I \to\mc F \to \li \to 0$ yields 
			\begin{equation}
				\label{5term}
				\xymatrix{
					0\ar[r]&H^1(\li ,\F)\ar[r]^{H^1(\phi )}&H^1(\mc F ,\F)\ar[r]^\alpha&
					H^1(I ,\F)^{\li }\ar[d]\\
					&&H^2(\mc F ,\F)&H^2(\li ,\F)\ar[l]\\}
			\end{equation}
			where the action of $\li $ on the $1$-cocycles $\alpha:I \to \F$ is given by $x\cdot\alpha(i)=\alpha([\tilde x,i])$, $\tilde x$ being any lift of $x\in \li $ to $\mc F $. 
			By construction, $H^1(\phi )$ is an isomorphism, and $H^2(\li ,\F)=0$, by hypothesis. 
			This implies that $H^1(I ,\F)^{\li }=0$. 
			As $I $ is an $\mathbb N$-graded ideal of $\mc F $, one has $H^1(I ,\F)=\hom_{\mbox{\small{Lie}}}(I ,\F)=\hom_{\F}(I /[I ,I ],\F)$. 
			Thus, $H^1(I ,\F)^{\li }=0$ implies $I /[\mc F ,I ]=0$, and therefore $I =[\mc F ,I ]$. 
			Now, suppose $n=\min\set{m\geq 0}{I_m\neq 0}$ is finite. One has $I_n=[I ,\mc F ]_n=\sum_{1\leq j<n}[I_j,\mc F_{n-j}]=0$, since $I $ is a graded ideal, whence $I =0$, and $\phi $ is injective.
		\end{proof}
	\end{thm}
	Notice that the cohomological dimension of subagebras does not exceed that of the Lie algebra. In particular, one may see Theorem \ref{NS} as a direct consequence of Theorem \ref{freecoh}. However, we decided to give the above proof that is purely combinatorial and it does not involve any cohomological arguement.
	
	In order to prove the Kurosh Theorem we first need to understand the algebra generated by some subalgebras of the factors $\mc A$ and $\mc B$ in the free product $\mc A\amalg\mc B$. The following 
	\begin{prop}{\label{subfreeproduct}}
		Let $\mc A$ and $\mc B$ be $\mathbb N$-graded $1$-generated $\F$-Lie algebras, and let $\mc H\subseteq \mc A\amalg\mc B$ be an $\mathbb N$-graded $1$-generated subalgebra of their free product, such that $\mc H=\gen{\mc H_1\cap\mc A,\mc H_1\cap \mc B}$. Then \[\mc H=\gen{\mc H_1\cap\mc A}\amalg\gen{\mc H_1\cap \mc B}.\]
		\begin{proof}
			Let $\mf X$ be a graded basis for $\mc H\cap\mc A$, and $\mf Y$ be a graded basis for $\mc H\cap\mc B$. Extend $\mf X$ and $\mf Y$ to graded bases $\mf X'$ for $\mc A$ and $\mf Y'$ for $\mc B$, respectively. Now, for every $x_1,x_2\in \mf X'$ there exist scalars $c_{x_1,x_2,z}\in \F$ such that \[[x_1,x_2]=\sum_{z\in \mf X'}c_{x_1,x_2,z}z.\]
			Similarly, for every $y_1,y_2\in\mf Y'$, one can write \[[y_1,y_2]=\sum_{w\in \mf Y'}d_{y_1,y_2,w}w.\]
			Define 
			\begin{equation}
				J=\left([x_1,x_2]-\sum_{z\in \mf X}c_{x_1,x_2,z}z,\ [y_1,y_2]-\sum_{w\in \mf Y}d_{y_1,y_2,w}w\ \vert\ x_i\in \mf X,\ y_j\in \mf Y\right)\end{equation}
			as an ideal of $ \li(\mf X\sqcup\mf Y)$, and
			\begin{equation}
				J'=\left([x_1,x_2]-\sum_{z\in \mf X'}c_{x_1,x_2,z}z,\ [y_1,y_2]=\sum_{w\in \mf Y'}d_{y_1,y_2,w}w\vert\ x_i\in \mf X',\ y_j\in \mf Y'\right)\end{equation} as an ideal of $\li(\mf X'\sqcup\mc Y')$.
			It is clear that $J=\li(\mf X\sqcup\mc Y)\cap J'$, where one considers the free Lie algebra as a subset of the free associative algebra.
			Now, the canonical epimorphism $\li(\mf X'\sqcup\mf Y')\to \mc A\amalg\mc B$ has kernel $J'$, and thus its restriction to $\li(\mf X\sqcup\mf Y)$ has kernel $J$, whence we get an isomorphism $\mc H\overset{\sim}{\to} \gen{\mc H_1\cap\mc A}\amalg\gen{\mc H_1\cap \mc B}.$
		\end{proof}
	\end{prop}
	
	The latter two results are the main ingredients for the Kurosh subalgebra theorem.
	
	\vspace{1cm}
	
	Let $V$ be a $\F$-subspace of the direct sum $A\oplus B$ of two finite dimensional $\F$-vector spaces $A$ and $B$. 
	Then one can find a distinguished basis $\mathcal B=\mathcal B_A\cup\mathcal B_B\cup\tilde{\mathcal B}$ for $V$, such that $\mathcal B_A$ is a basis for $V\cap A$, $\mathcal B_B$ is a basis for $V\cap B$, and $\tilde{\mathcal B}$ satisfies the following property: 
	The projections $\pi_A(x)$, $x\in \tilde{\mathcal B}$, and the set $\mathcal B_A$ form a linearly independent set in $A$ (and similarly for $B$).
	Now, since $\pi_A(\tilde{\mathcal B})\cup \mathcal B_A$ is a linearly independent set in $A$, one can thus complete this set to a basis $\mathcal B'_A$ for the whole $A$.
	
	Indeed, let the set $\tilde{\mathcal B}=\{c_i=a_i'+b_i'\}$ ($a_i'\in A\setminus\{0\},b_i'\in B\setminus\{0\}$) complete the set $\mathcal B_A\cup \mathcal B_B$ to a basis for $V$. Suppose that there are scalars $\alpha_i,\beta_j\in\F$, such that $\sum_i\alpha_ia_i'+\sum_j\beta_j a_j=0$. Hence the element $\sum_i\alpha_i c_i+\sum_j\beta_j a_j=\sum_i\alpha_i b_i$ lies in $V\cap B$. Since $\mbox{Span}(\mathcal B_B,\mathcal B_A)\cap \mbox{Span}(\tilde{\mathcal B})=0$, one has $\alpha_i=0$, and thus $\beta_j=0$. In case $V\cap A=0$, one can set $\mathcal B_A=\emptyset$.\\
	The above discussion also holds for $B$.

	Let us now specialize the previous consideration for the $1$-generated subalgebras of the free product $\mc A\amalg\mc B$ of two $1$-generated $\mathbb N$-graded $\F$-Lie algebras $\mc A$ and $\mc B$. 
	
	Let $W\leq \mc A_1\oplus\mc B_1$ be a subspace such that $W\cap\mc A_1=W\cap\mc B_1=0$. Let $(c_i)$ be a basis for $W$. One can write $c_i=a_i+b_i$ in a unique way, such that $a_i\in\mc A_1$ and $b_i\in \mc B_1$. Moreover, $(a_i)$ is a linearly independent set in $\mc A_1$, and $(b_i)$ is a linearly independent set in $\mc B_1$. Consider the envelopes $A=\mc U(\mc A)$ and $ B=\mc U(\mc B)$. Therefore, \[\mc U(\mc A\amalg\mc B)=A\amalg B=\F\oplus\overbrace{ A_1\oplus B_1}^{\text{degree}\ 1}\oplus \overbrace{A_2\oplus (A_1\otimes B_1)\oplus (B_1\otimes A_1)\oplus B_2}^{\text{degree}\ 2}\oplus \dots\]
	If an arbitrary element $\sum_{i,j}\alpha_{ij}c_ic_j$ of $\gen{W}_2$ lies in $A$, then $\sum_{i,j}\alpha_{i,j}(a_ib_j+b_ia_j)\in A_2$. But the latter belongs to a complementary space of $A_2$ in $(A\amalg B)_2$, namely $A_1\otimes B_1\oplus B_1\otimes A_1$, and thus $\alpha_{ij}=0$.
	The same holds for greater degrees. 
	
	Let $x=\sum_{\vert I\vert=n}\alpha_Ic_I\in \gen{W}_{\F-\text{alg}}\cap A_n$, where $c_I=c_{i_1}\cdots c_{i_n}$, for $I=(i_1,\dots,i_n)$, and $\alpha_I\in\F$. 
	Note that, if $n$ is even, then the elements $a_{i_1}b_{i_2}\dots a_{i_{n-1}}b_{i_n}$ $(I=(i_1,\dots,i_n))$ are linearly independent in \[\overbrace{A_1\otimes B_1\otimes \dots \otimes A_1\otimes B_1}^{n\text{ terms}}\] and thus $\alpha_I=0$, for all $I$. For $n$ odd, the same holds. 
	We have thus proven
	\begin{prop}{\label{nullinters}}
		Let $\mc A$ and $\mc B$ be $1$-generated $\mathbb N$- graded $\F$-Lie algebras, and let $\mc H\leq \mc A\amalg \mc B$ be a $1$-generated Lie subalgebra.
		Then the following are equivalent: 
		\begin{enumerate}
			\item $\mc H_1\cap\mc A_1=0$, and 
			\item $\mc H\cap \mc A=0$.
		\end{enumerate}
	\end{prop}

	\begin{lem}{\label{freeres}}
		Let $\li$ be an $\mathbb N$-graded $\F$-Lie algebra, and let $\mc A,\mc H\leq \li$ be $\mathbb N$-graded Lie subalgebras of $\li$ satisfying $\mc A\cap \mc H=\{0\}$. 
		Then $\mc U(\li)$ is a free $(\mc U(\mc H),\mc U(\mc A))$-bimodule. 
		In particular, $\res^{\li}_{\mc H}\ind^{\li}_{\mc A}(\F)$ is a free left $\mc U(\mc H)$-module.
		\begin{proof}
			Set $H=\mc U(\mc H)$ and $A=\mc U(\mc A)$.\\
			The first assertion follows by the PBW theorem. 
			Indeed, if $\{a_i\}$ and $\{h_j\}$ are basis for $\mc A$ and $\mc H$, respectively, then there is a basis $\{a_i,h_j,l_k\}$ for $\mc L$. 
			Therefore, $\mc U(\li)$ has basis $\{\prod_j h_i\prod_k l_k\prod_i a_i\}$ where $i,j,k$ range over ordered sets of indices. 
			This means that \[\mc U(\li)=\bigoplus_{K'} (H\otimes A^{op})\prod_{k\in K'} l_k,\] where $K'$ ranges over some finite sets of indices $k$, is a free $(H\otimes A^{op})$-module, i.e. a free $(H,A)$-bimodule.
		\end{proof}
	\end{lem}
	\begin{thm}{\label{subfree}}
		Let $\mc A$ and $\mc B$ be $\mathbb N$-graded $1$-generated $\F$-Lie algebras, and let $W\subseteq \mc A_1\oplus\mc B_1$ be a subspace satisfying ${W}\cap \mc A={W}\cap \mc B=0$. Then the $1$-generated Lie subalgebra $\mc F=\gen{W}\subseteq \mc A\amalg\mc B$ is a free $\F$-Lie algebra. 
		\begin{proof}
			Put $\li=\mc A\amalg\mc B$.
			By Theorem \ref{thm:mayervietoris}, one has an exact sequence \begin{equation*}
				0\to \mc U(\li)\to \ind_{\mc A}^{\li}(\F)\oplus\ind_{\mc B}^{\li}(\F)\to \F\to 0
			\end{equation*} of $\mathbb N_0$-graded left $\mc U(\li)$-modules.
			Now, by Proposition \ref{nullinters}, $\mc F\cap \mc A=\mc F\cap\mc B=0$, and thus, by Lemma \ref{freeres}, $\res_{\mc F}^{\li}\ind_{\mc A}^{\li}(\F)$ and $\res_{\mc F}^{\li}\ind_{\mc B}^{\li}(\F)$ are free $\mc U(\mc F)$-modules. 
			Therefore, 
			\begin{equation*}
				0\to \res_{\mc F}^{\li}\mc U(\li)\to \res_{\mc F}^{\li}\ind_{\mc A}^{\li}(\F)\oplus\res_{\mc F}^{\li}\ind_{\mc B}^{\li}(\F)\to \res_{\mc F}^{\li}\F\to 0
			\end{equation*}
			is a free resolution of $\F$ over $\mc U(\mc F)$, and hence $\text{cd}(\mc F)\leq 1$. This proves that $\mc F$ is free, by Theorem \ref{freecoh}. 
		\end{proof}
	\end{thm}
	
	\begin{lem}\label{indep sum}
		Let $\mc A$ and $\mc B$ be two quadratic Lie algebras. Let $a_i\in \mc A_1$ and let $\{b_i\}\subset \mc B_1$ be an independent set. If $\sum_i a_ib_i=0$ in the free product $\mc U(\mc A\amalg \mc B)$, then $a_i=0$, $\forall i$.
		\begin{proof}
			This follows from the fact that there is no non-trivial defining relation for $\mc U(\mc A\amalg\mc B)$ involving elements in $\mc A_1\mc B_1$. 
			Indeed, since both the Lie algebras are quadratic, there are relations $r_{\mc A}\in \mc A_1^{\otimes 2},\ r_{\mc B}\in \mc B_1^{\otimes 2}$ such that $\sum_i a_ib_i=r_{\mc A}+r_{\mc B}$.
		\end{proof}
	\end{lem}
We are finally ready to prove the main result of this paper.
	\begin{thm}[Kurosh' subalgebra theorem] 
		
		Let $\mc A$ and $\mc B$ be two Bloch-Kato $\F$-Lie algebras, and let $\mc H \subseteq \mc A\amalg \mc B$ be a $1$-generated subalgebra. Then \begin{equation}
			\mc H\simeq \gen{\mc H_1\cap \mc A}\amalg \gen{\mc H_1\cap \mc B}\amalg \mc F
		\end{equation} where $\mc F$ is a free Lie algebra generated by any distinguished subspace $W\subseteq \mc A_1\oplus\mc B_1$ such that $\mc H_1=W\oplus (\mc H_1\cap \mc A)\oplus(\mc H_1\cap \mc B)$.
		\begin{proof}
			Notice first that $\mc H$ is a Koszul Lie algebra, for $\li$ is Bloch-Kato by Theorem \ref{et=>bk}.
			Decompose $\mc H_1$ into the direct sum $(\mc H_1\cap \mc A)\oplus(\mc H_1\cap \mc B)\oplus W$. 
			Then, $\mc F=\gen{W}$ is a free Lie subalgebra of $\mc H$  by Theorem \ref{subfree}, and $\mc H=\gen{W,\mc H_1\cap \mc A,\mc H_1\cap \mc B}$.
			By Proposition \ref{subfreeproduct}, one has \[\mc H=\gen{\mc F,\gen{\mc H_1\cap \mc A}\amalg\gen{\mc H_1\cap \mc B}},\] where $\mc F$ is the free Lie algebra on $W$. 
			Let $\mc P=\gen{\mc H_1\cap \mc A}\amalg\gen{\mc H_1\cap \mc B}$. Thus, in light of Theorem \ref{mayervietoris}, it suffices to prove that $\mc U(\mc H)$ is the kernel of the canonical mapping \[\varepsilon=\varepsilon^{\mc F}-\varepsilon^{\mc P}:\ind_{\mc F}^{\mc H}\F\oplus\ind_{\mc P}^{\mc H}\F\to \F.\]
			
			We first prove that the image of the natural map $\mc U(\mc H)\to\ind_{\mc F}^{\mc H}\F\oplus\ind_{\mc P}^{\mc H}\F$ is the kernel of $\varepsilon$. 
			Let $(x\otimes 1,y\otimes 1)\in \ind_{\mc F}^{\mc H}\F\oplus\ind_{\mc P}^{\mc H}\F$ be a homogeneous element in the kernel of $\varepsilon$. 
			This means that $x-y$ has positive degree in $\mc U(\mc H)$. 
			Since $\mc H$ is generated by $\mc P$ and $\mc F$, there are elements $p,f\in\mc U(\mc H)$ such that $x-y=p-f$, and $p$ and $f$ can be written as combinations of products ending with elements in $\mc P_1$, and $\mc F_1$ respectively.  
			Thus, \[x+f\in\mc U(\mc H)\longmapsto ((x+f)\otimes 1,(y+p)\otimes 1)=(x\otimes 1,y\otimes 1).\]
			
			Now, it remains to prove that the map $\mc U(\mc H)\to \ind_{\mc F}^{\mc H}\F\oplus\ind_{\mc P}^{\mc H}\F$ is injective. 
			Let $x\in \mc U(\mc H)_2$ such that $f(x)=(x\otimes 1,x\otimes 1)=0$. 
			It follows that $x\in \mc U(\mc H)\mc P_1\cap\mc U(\mc H)\mc F_1$, namely there are elements $ x_i,y_j,z_k\in\mc H_1$ such that \[x=\sum x_ia_i+\sum y_jb_j=\sum z_k (a'_k+b'_k)\] where $(a_i,b_j)$ and $(a'_k+b'_k)$ are $\F$-basis respectively of $\mc P_1$ and $\mc F_1=W$. 
			Moreover, the set $(a_i,b_j,a'_k,b'_k)$ is linearly independent.
			Let $(\argu)^{\mc A}:\mc H_1\to\mc A_1$ and $(\argu)^{\mc B}:\mc H_1\to\mc B_1$ be the projections ($\mc H_1\subseteq \mc A_1\oplus\mc B_1$).
			Since $\ul_2=\mc U(\mc A)_2\oplus \mc A_1\mc B_1\oplus\mc B_1\mc A_1\oplus\mc U(\mc B)_2$, it follows that \begin{align}
				\sum x_i^{\mc A}a_i=\sum z_k^{\mc A}a'_k\\
				\label{second}	\sum x_i^{\mc B}a_i=\sum z_k^{\mc B}a'_k\\
				\label{third}	\sum y_j^{\mc A}b_j=\sum z_k^{\mc A}b'_k\\
				\sum y_j^{\mc B}b_j=\sum z_k^{\mc B}b'_k
			\end{align}
			
			Consider the equations (\ref{second}), (\ref{third}). 
			From (\ref{second}), we recover a relation for $\ul$ involving elements in $\mc B_1\mc A_1$. Since the set $\{a_i,a'_k\}$ is linearly independent, the relation is not trivially satisfied but when $x_i^{\mc B}=z_k^{\mc B}=0$.
			
			For the equation (\ref{third}), the same holds, proving that $z_k=0$. Therefore, $x=0$.

			This shows that the map $\mc U(\mc H)\to\ind_{\mc F}^{\mc H}\F\oplus\ind_{\mc P}^{\mc H}\F$ is injective in degree $2$. 
			By construction, it is also injective in degree $< 2$. 
			By Corollary \ref{inj koszul}, the map is injective in all the degrees.

		\end{proof}
	\end{thm}
\begin{rem}
	If $\mc U(\mc H)\to \ind_{\mc F}^{\mc H}\F\oplus\ind_{\mc P}^{\mc H}\F$ is injective, since both the modules are Koszul, also the cokernel is Koszul, and generated in degree $0$. But the cokernel is $\F$ in degree $0$ and it is $0$ in degree $1$, proving that it must be the trivial module. 
	This argument should give another proof for the first part of the proof.
	
	Notice that the latter proof may be adapted to give a less combinatorial proof for Proposition \ref{subfreeproduct}.
\end{rem}

\end{document}